\documentclass[a4paper,twoside,reqno]{amsart}

\usepackage[pagewise]{lineno}
\usepackage[utf8]{inputenc}

\usepackage{authblk}
\usepackage{hyperref}
\hypersetup{urlcolor=blue, citecolor=red}
\usepackage{amsmath,amsthm,amssymb,xcolor,bbm,mathrsfs}
\usepackage[english]{babel}
\usepackage{fancyhdr}
\newcommand{\norm}[1]{\|#1\|}
\newcommand{\n}{\hspace*{-6pt}}
\renewcommand{\c}{\mathbf{c}}
\renewcommand{\r}{\mathbf{r}}
\newcommand{\E}{\mathbf{E}}

\DeclareMathOperator{\diam}{diam}

\DeclareMathOperator{\unif}{uniform}

\makeatletter
\@namedef{subjclassname@1991}{2020 Mathematics Subject Classification}
\makeatother

\title{Mixed Deffuant Dynamics}
\author{Hsin-Lun Li}

\date{}
\email{hsinlunl@asu.edu}

\theoremstyle{definition}
\newtheorem{theorem}{Theorem}

\newtheorem{lemma}[theorem]{Lemma}
\newtheorem{corollary}[theorem]{Corollary}

\newtheorem{example}[theorem]{Example}

\begin{document}

\allowdisplaybreaks

\thispagestyle{firstpage}
\maketitle
\begin{center}
    Hsin-Lun Li
\end{center}

\begin{abstract}
The original Deffuant model consists of a finite number of agents whose opinion is a number in $[0,1]$. Two socially connected agents are uniformly randomly selected at each time step and approach each other at a rate $\mu\in [0,1/2]$ if and only if their opinions differ by at most some confidence threshold $\epsilon>0$. In this paper, we consider a variant of the Deffuant model, namely the mixed model, where the convergence parameter $\mu$ and the social relationship can vary over time. We investigate circumstances under which asymptotic stability holds or a consensus can be achieved. Also, we derive a nontrivial lower bound for the probability of consensus which is independent of the number of agents.
\end{abstract}

\section{Introduction}
The original Deffuant model comprises a finite number of agents whose opinion is a number in $[0,1]$.  Two socially connected agents are uniformly randomly selected at each time step and approach each other at a rate $\mu\in[0,1/2]$ if and only if their opinions differ by at most some confidence threshold $\epsilon>0$. Interpreting in graph, two types of graphs are involved--the social graph and the opinion graph. The social graph depicts a social relationship, such as a friendship, whereas the opinion graph indicates an opinion relationship. Both graphs have a common vertex set in which each vertex stands for an agent. The edge of the social graph exists if and only if the two vertices are socially connected, whereas the edge of the opinion graph exists if and only if the two vertices are at a distance of at most $\epsilon$ apart. Two vertices are neighbors if and only if the edge connecting the two exists. The social relationship is assumed constant. The model is as follows:
$$\begin{array}{rcl}
 \displaystyle  x_i(t+1)  &\n=\n & x_i(t) + \mu(x_j(t)-x_i(t))\mathbbm{1}\{\norm{x_i(t)-x_j(t)}\leq\epsilon \}\vspace{2pt} \\
\displaystyle   x_j(t+1)  &\n=\n & x_j(t) + \mu(x_i(t)-x_j(t))\mathbbm{1}\{\norm{x_i(t)-x_j(t)}\leq\epsilon \}  
\end{array}$$
where
$$\begin{array}{l}
  \displaystyle   x_i(t)=\hbox{opinion of agent $i$ at time $t$},\vspace{2pt}  \\
  \displaystyle   \hbox{Agents}\ i\ \hbox{and}\ j\ \hbox{are socially connected and selected at time}\ t, \vspace{2pt}\\
 \displaystyle    \mu=\hbox{convergence parameter}. 
\end{array}$$
The authors in~\cite{Deffuant} assume that the social graph is constant and connected and demonstrate that asymptotic stability holds for the Deffuant model in all finite dimensional real number systems. Now, we consider a variant of the Deffuant model, namely the mixed model, where the opinion space can be high dimensional, say $\mathbf{R^d}$, and the social relationship among all agents and the convergence parameter can vary over time. That is, all agents can decide their social relationship with each other and the two selected social neighbors can decide their approach rate at all times. The mixed model is as follows:
$$\begin{array}{rcl}
 \displaystyle  x_i(t+1)  &\n=\n & x_i(t) + \mu(t)(x_j(t)-x_i(t))\mathbbm{1}\{\norm{x_i(t)-x_j(t)}\leq\epsilon \}\vspace{2pt} \\
\displaystyle   x_j(t+1)  &\n=\n & x_j(t) + \mu(t)(x_i(t)-x_j(t))\mathbbm{1}\{\norm{x_i(t)-x_j(t)}\leq\epsilon \}  
\end{array}$$
where
$$\begin{array}{l}
 \displaystyle    x_i(t)=\hbox{opinion of agent $i$ at time $t$},\vspace{2pt}  \\
 \displaystyle    \hbox{Agents}\ i\ \hbox{and}\ j\ \hbox{are socially connected and selected at time}\ t, \vspace{2pt}\\
     \mu(t)=\hbox{convergence parameter at time}\ t. 
\end{array}$$ 
In particular, the mixed model reduces to the Deffuant model if the social graph and convergence parameter are constant, i.e., $G(t)=G$ and $\mu(t)=\mu$ at all times. Say vertices $i$ and $j$ are opinion-connected at time $t$, or edge $(i,j)$ is opinion-connected at time $t$ if $\norm{x_i(t)-x_j(t)}\leq\epsilon.$  Denote $G=(V,E)$ as a graph $G$ with vertex set $V$ and edge set $E$.  The social graph at time $t$ and the opinion graph at time $t$ are
$$G(t)=(V,E(t))\quad \hbox{and}\quad \mathscr{G}(t)=(V,\mathscr{E}(t))$$
where
$$\begin{array}{rcl}
 \displaystyle   E(t) &\n=\n& \{(i,j):\hbox{vertices}\ i\ \hbox{and}\ j\ \hbox{are socially connected}\} \vspace{2pt}\\
 \displaystyle   \mathscr{E}(t)&\n=\n& \{(i,j): \norm{x_i(t)-x_j(t)}\leq\epsilon\}.
\end{array}$$
The opinion update occurs if and only if the selected social edge is opinion-connected. Therefore, we define a \emph{profile} at time $t$ as the intersection of a social graph and an opinion graph at time $t$, namely $G(t)\cap \mathscr{G}(t)$. Note that the edges in~$G(t)\cap\mathscr{G}(t)$ are both opinion-connected and socially connected. Prior to showing asymptotic stability holds for the mixed model, we introduce some terms. A graph is $\delta$-\emph{trivial} if any two vertices are at a distance of at most $\delta$ apart.  The \emph{convex hull} generated by~$v_1, v_2, \ldots, v_n \in \mathbf{R^d}$ is the smallest convex set containing~$v_1, v_2, \ldots, v_n$, i.e.,
 $$ C (\{v_1, v_2 \ldots, v_n \}) = \{v : v = \sum_{i = 1}^n \lambda_i v_i \ \hbox{where} \ (\lambda_i)_{i = 1}^n \ \hbox{is stochastic} \}. $$
The author in~\cite{mhk} has indicated that an opinion graph is nonincreasing and preserves~$\delta$-triviality for all~$\delta>0$ over time.

\section{The model}
\begin{lemma}\label{inq}
For all $c\in \mathbf{R^d}$,
$$\begin{array}{rcl}
  \displaystyle  \norm{x_i(t+1)-c}+\norm{x_j(t+1)-c} &\n\leq\n& \norm{x_i(t)-c}+\norm{x_j(t)-c}, \vspace{2pt}\\
  \displaystyle \norm{x_i(t+1)-c}+\norm{x_j(t+1)-c} &\n\leq\n& \norm{x_i(t)-c}+\norm{x_j(t)-c} \vspace{2pt}\\
   &&-2\norm{x_i(t)-x_i(t+1)}+2\norm{c_{ij}(t)-c}
\end{array}$$
where $c_{ij}(t)=(x_i(t)+x_j(t))/2$.
\end{lemma}

\begin{proof}
It is clear that the two inequalities hold if $(i,j)\in E(t)$ is not selected at time $t$. Assume that $(i,j)\in E(t)$ is selected at time $t$, then,
\begin{align*}
   \norm{x_i(t+1)-c}=\norm{(1-\mu(t))(x_i(t)-c)+\mu(t)(x_j(t)-c)}
\end{align*}
therefore by the triangle inequality, we derive the first inequality. To show the second inequality, applying the triangle inequality,
\begin{align*}
   \norm{x_i(t+1)-c} &\leq \norm{x_i(t+1)-c_{ij}(t)}+\norm{c_{ij}(t)-c}\\
   & = \norm{x_i(t)-c_{ij}(t)}-\norm{x_i(t)-x_i(t+1)}+\norm{c_{ij}(t)-c}.
\end{align*}
Also,
$$\begin{array}{rcl}
  \displaystyle  \norm{x_i(t)-c_{ij}(t)}+\norm{x_j(t)-c_{ij}(t)}&\n=\n&\norm{x_i(t)-x_j(t)} \leq \norm{x_i(t)-c}+\norm{x_j(t)-c},\vspace{2pt}   \\
\displaystyle  \norm{x_i(t)-x_i(t+1)}&\n=\n& \norm{x_j(t)-x_j(t+1)}
\end{array}$$
therefore we have the second inequality.
\end{proof}

For all $c\in\mathbf{R^d}$, let $$Z_c(t)=\sum_{i\in V}\norm{x_i(t)-c}.$$ From Lemma \ref{inq}, $Z_c(t)$ is nonincreasing with respect to $t$ and
\begin{equation}\label{sup}
  Z_c(t)-Z_c(t+1)\geq 2(\norm{x_i(t)-x_i(t+1)}-\norm{c_{ij}(t)-c})  
\end{equation}
for $(i,j)\in E(t)$ selected at time $t$. Let $\bar{B}(x,r)=\{v:\norm{x-v}\leq r\}$, a closed ball centered at $x$ of radius $r$. Next, we study circumstances under which The distance between any two social and opinion neighbors is at most $\delta$ in finite time for all $\delta>0$. 

\begin{theorem}\label{fin}
Assume that $\liminf_{t\to\infty}\mu(t)>0$. Then for all $\delta>0$, all opinion edges of a profile are of length at most $\delta$ in finite time almost surely, i.e.,
$$\tau_{\delta}:=\inf\{t\geq 0:\norm{x_i(t)-x_j(t)}\leq\delta\ \hbox{for all}\ (i,j)\in E(t)\cap\mathscr{E}(t)\}<\infty\ \hbox{almost surely}.$$
\end{theorem}

\begin{proof}
If $E(t)$ is empty for some $t\geq0$, then we are done. Assume that $E(t)$ is nonempty for all $t\geq 0$. Since $\liminf_{t\to\infty}\mu(t)>0$, there is $(u_k)_{k\geq 0}\subset\mathbf{N}$ increasing such that $\mu(u_k)\geq \tilde{\mu}>0$ for some $\tilde{\mu}$ and for all $k\geq 0$. Assume that this is not the case. Then, there are $$(t_k)_{k\geq 0}\subset (u_k)_{k\geq 0}\quad \hbox{and}\quad (i,j)\in E(t_k)\cap \mathscr{E}(t_k)$$ such that $$\norm{x_i(t_k)-x_j(t_k)}>\delta\ \hbox{for all}\ k\geq0.$$ 
Let $E_k$ be the event that the edge $(i,j)$ is selected at time $k$. Then, all $E_{t_k}$ are independent and
$$\sum_{k\geq0}P(E_{t_k})\geq \sum_{k\geq 0} 1/\binom{|V|}{2}=\infty $$ therefore by the second Borel-Cantelli Lemma, $P(\limsup_{k\to \infty}E_{t_k})=1.$ Hence, there is $(s_{\ell})_{\ell\geq 0}\subset (t_k)_{k\geq 0}$ such that $E_{s_{\ell}}$ holds for all $\ell\geq 0$. From \eqref{sup},
$$Z_c(s_{\ell})-Z_c(s_{\ell}+1)> 2(\mu(s_{\ell}) \delta-\norm{c_{ij}(s_{\ell})-c})\geq 2(\Tilde{\mu} \delta-\norm{c_{ij}(s_{\ell})-c})\geq \Tilde{\mu} \delta$$
for all $c\in \bar{B}(c_{ij}(s_{\ell}),\Tilde{\mu}\delta/2)$. Since $C\big((x_i(t))_{i\in V}\big)$ is bounded by $C\big((x_i(0))_{i\in V}\big)$, it can be covered by finitely many cubes of length $\Tilde{\mu}\delta/(2\sqrt{2})$, therefore existing a cube containing infinitely many $c_{ij}(s_{\ell})$. Hence, there is $I\subset (s_{\ell})_{\ell\geq 0}$ infinite with $\bigcap_{t\in I}\bar{B}(c_{ij}(t),\Tilde{\mu}\delta/2)$ containing some cube of length $\Tilde{\mu}\delta/(2\sqrt{2})$. Pick $c_0\in \bigcap_{t\in I}\bar{B}(c_{ij}(t),\Tilde{\mu}\delta/2)$,
$$\infty> \sum_{t\in I}(Z_{c_0}(t)-Z_{c_0}(t+1))\geq \sum_{t\in I}\Tilde{\mu}\delta=\infty,\ \hbox{a contradiction}.$$
\end{proof}

It turns out that any two social and opinion neighbors are at a distance of at most $\delta$ apart in finite time as long as there are infinitely many times that the two selected social neighbors approach each other at a rate no less than some positive number. It follows that asymptotic stability holds for the mixed model if the two selected social neighbors at all times approach each other at a rate no less than some positive number.

\begin{corollary}\label{asy}
Assume that $\inf_{t\geq 0}\mu(t)>0$. Then, all edges of a profile are of length at most $\delta$ after some almost surely finite time for all $\delta>0$, i.e.,
$$\norm{x_i(t)-x_j(t)}\leq \delta\ \hbox{for all}\ (i,j)\in E(s)\cap\mathscr{E}(s)\ \hbox{for some}\ t\ \hbox{and for all}\ s\geq t.$$
Therefore, asymptotic stability holds for the mixed model.
\end{corollary}

\begin{proof}
Let $\mathscr{A}=\{t\geq 0:E(t)\neq \varnothing\}$. If $\mathscr{A}$ is finite, then we are done. Assume that $\mathscr{A}$ is infinite and that there is an edge $(i,j)$ of a profile of length larger than $\delta$ for infinitely many times for some $\delta>0$. Let $\mu=\inf_{t\geq 0}\mu(t)$, $$\mathscr{B}=\{t\in\mathscr{A}: \norm{x_i(t)-x_j(t)}>\delta\}\ \hbox{and}\ E_k=\{(i,j)\ \hbox{is selected at time}\ k\}.$$ By the second Borel-Cantelli Lemma, $$\sum_{t\in\mathscr{B}}P(E_t)\geq \sum_{t\in\mathscr{B}}1/\binom{|V|}{2}=\infty\quad \hbox{implies}\quad P(\limsup_{t\in\mathscr{B},t\to\infty}E_t)=1.$$ Therefore, almost surely infinitely many $E_t$ hold for $t\in \mathscr{B}$, say $(t_k)_{k\geq 0}$. Similarly as the proof in Theorem~\ref{fin}, there is $I\subset (t_k)$ infinite with $$\bigcap_{t\in I}\bar{B}(c_{ij}(t),\mu\delta/2)\ \hbox{containing some cube of length}\ \mu\delta/(2\sqrt{2}).$$ Pick $c_0\in \bigcap_{t\in I}\bar{B}(c_{ij}(t),\mu\delta/2)$,
$$\infty> \sum_{t\in I}(Z_{c_0}(t)-Z_{c_0}(t+1))\geq \sum_{t\in I}\mu\delta=\infty,\ \hbox{a contradiction}.$$
To show the asymptotic stability of $(x_i)_{i\in V}$, let the convex hull of a component~$G$ be
 $$ Cv(G) = C(\{x_j : j \in V(G) \}). $$ All edges of a profile are of length at most $\delta/|V|$ after some time $t$ for all $\delta>0$.  For all $s\geq t$ and $i\in V$, $x_i(s)\in Cv(G)$ for some component $G$ of $G(s)\cap \mathscr{G}(s)$ and $$\diam(Cv(G))\leq(|V|-1)\delta/|V|<\delta.$$ Therefore, the asymptotic stability holds for the mixed model.
\end{proof}

\section{Probability of consensus}
Assume that all initial opinions $x_i(0)$, $i\in V$, are independent and identically distributed random variables on the bounded convex opinion space $\Delta\subset\mathbf{R^d}$, and let $X$ be the random variable with distribution
$$P(X\in B)=P(x_i(0)\in B)\ \hbox{for all}\ i\in V\ \hbox{and all Borel sets}\ B\subset \Delta.$$ Let $\c=\arg\inf_{c\in \mathbf{R^d}}\sup_{z\in\Delta}\norm{z-c}$ be the center of $\Delta$ and $\r=\sup_{z\in \Delta}\norm{z-\c}$. Then, $\c$ uniquely exists and $d/2\leq \r\leq \sqrt{3}d/2$ for $d=\diam(\Delta)$. Let $\mathscr{C}$ be the collection of all sample points that lead to a consensus, i.e., $$\mathscr{C}=\{\limsup_{t\to\infty}\max_{i,j\in V}\norm{x_i(t)-x_j(t)}=0\}.$$ Assume that the social graph is connected for infinitely many times. Define 
\begin{align*}
   T_{\delta} =&\inf\{t\geq 0: G(t)\cap\mathscr{G}(t)\ \hbox{is connected and all edges of}\ G(s)\cap\mathscr{G}(s)\ \\
   &\hspace{20pt}\hbox{are of length at most $\delta$ for all}\ s\geq t\}, 
\end{align*}
 the earliest time that a profile is connected, all edges of which are of length at most $\delta$ afterward. Thus, $T_{\delta}$ is almost surely finite if $\inf_{t\geq 0}\mu(t)>0$. The following lemma indicates circumstances under which a consensus can be achieved.

\begin{lemma}\label{consensus}
Assume that $\inf_{t\geq 0}\mu(t)>0$ and $\epsilon>\r$. Then, there is $\delta$ such that $$A=\{\norm{x_i(T_{\delta})-\c}\leq\epsilon-\r-\eta\ \hbox{for some}\ i\in V\}\subset\mathscr{C}\ \hbox{for all}\ 0<\eta<\epsilon-\r.$$
\end{lemma}

\begin{proof}
Take $\delta=\eta/|V|$. For all $j\in V$, $\norm{x_j(T_{\delta})-c}\leq r$ therefore by the triangle inequality, $\norm{x_i(T_{\delta})-x_j(T_{\delta})}\leq \epsilon-\eta$. Thus, $(i,j)\in \mathscr{E}(T_{\delta})$. In particular, $$\norm{x_i(T_{\delta})-x_j(T_{\delta})}\leq \delta\quad \hbox{if}\quad (i,j)\in E(T_{\delta})\cap\mathscr{E}(T_{\delta}).$$ We claim that $\norm{x_j(T_{\delta})-x_k(T_{\delta})}\leq \delta$ for all $(j,k)\in E(T_{\delta})$. For all $k\in V$, there is an $i,k$-path in $G(T_{\delta})$, say $i=i_0 i_1\ldots i_{n-1}i_n=k$, of length at most $|V|-1$. Argue by induction that 
$$\norm{x_{i_\ell}(T_{\delta})-x_{i_{\ell-1}}(T_{\delta})}\leq\delta\ \hbox{for all}\ \ell=1,2,\ldots,n.$$
Clearly, $\norm{x_{i_1}(T_{\delta})-x_{i_0}(T_{\delta})}\leq\delta$. Suppose that it is true for $\ell=1,2,\ldots,m-1$. Assume by contradiction that $(i_m,i_{m-1})\notin \mathscr{E}(T_{\delta})$. Then, $\norm{x_{i_m}(T_{\delta})-x_{i_{m-1}}(T_{\delta})}>\epsilon$. Thus, by the triangle inequality,
\begin{align*}
    &\norm{x_{i_m}(T_{\delta})-x_i(T_{\delta})}\\
    &\geq\norm{x_{i_m}(T_{\delta})-x_{i_{m-1}}(T_{\delta})}-\norm{x_{i_{m-1}}(T_{\delta})-x_{i_{m-2}}(T_{\delta})+\ldots+x_{i_1}(T_{\delta})-x_i(T_{\delta})}\\
    &>\epsilon-(m-1)\delta>\epsilon-\eta,\ \hbox{a contradiction.}
\end{align*}
For all edges $(j,k)\in E(T_{\delta})$, we can always construct a path starting from $i$ that contains $(j,k)$, therefore verifying the claim. Also, via the triangle inequality and connectedness of $G(T_{\delta})$,
$$\norm{x_j(T_{\delta})-x_k(T_{\delta})}\leq (|V|-1)\delta<\eta<\epsilon\quad \hbox{for all}\quad j,k\in V.$$
An opinion graph preserves $\epsilon$-triviality; therefore $G(s)\cap\mathscr{G}(s)=G(s)$ for all $s\geq T_{\delta}$. Also, a social graph is connected for infinitely many times; hence $G(s)$ is connected for infinite many $s\geq T_{\delta}$. From Corollary~\ref{asy}, a consensus is achieved eventually.
\end{proof}

It follows that the probability of consensus has a nontrivial lower bound if the two selected social neighbors approach each other at a rate no less than some positive number at all times.

\begin{theorem}
Assume that $\inf_{t\geq 0}\mu(t)>0$ and $\epsilon>\r$. Then, 
$$P(\mathscr{C})\geq  1-\frac{\E(\norm{X-\c})}{\epsilon-\r}.$$
In particular, $P(\mathscr{C})=1$ if~$\epsilon\geq\diam(\Delta)$.
\end{theorem}

\begin{proof}
Via Lemma~\ref{consensus}, there is $\delta$ such that $$A=\{\norm{x_i(T_{\delta})-\c}\leq\epsilon-\r-\eta\ \hbox{for some}\ i\in V\}\subset\mathscr{C}\ \hbox{for all}\ 0<\eta<\epsilon-\r.$$ Since $Z_\c(t)$, $t\geq 0$, is a bounded supermartingale and $T_{\delta}$ is almost surely finite, by the optional stopping theorem and conditional expectation restricting to~$A^c$, 
 \begin{align*}
     |V|\ \E(\norm{X-\c})= \E(Z_\c(0))&\geq \E(Z_\c(T_{\delta}))\\
     &\geq \E(Z_c(T_{\delta})|A^c)P(A^c)\geq |V| (\epsilon-\r-\eta)P(A^c)
 \end{align*}
 Therefore, $$P(\mathscr{C})\geq P(A)\geq 1-\frac{\E(\norm{X-\c})}{\epsilon-\r-\eta}.$$
 Letting $\eta\to 0$, we have
$$P(\mathscr{C})\geq 1-\frac{\E(\norm{X-\c})}{\epsilon-\r}.$$
\end{proof}

It turns out that the lower bound for the probability of consensus is independent of the size and topology of a graph and the convergence rate, but depends on the confidence threshold, the radius, the norm, the distribution of an initial opinion and the center. In particular for the Deffuant model with $\mu(t)=\mu$ at all times, $P(\mathscr{C})=P(\hbox{a consensus is achieved at time $0$})$ if $\mu=0$; else $P(\mathscr{C})\geq 1-\E(\norm{X-\c})/(\epsilon-\r)$. The following is a demonstration of the probability of consensus given that~$\inf_{t\geq 0}\mu(t)>0$, $\norm{\,}$ is the Euclidean norm and $X$ is uniformly distributed on the opinion space~$\Delta$, denoted by $X=\unif(\Delta)$.

\begin{example}
For~$\Delta=B(\Vec{0},r)$, the open ball centered at~$\Vec{0}$ of radius~$r$, we have~$\c=\Vec{0}$, $\r=r$ and $\norm{X-\Vec{0}}=\unif\big((0,r)\big)$, therefore $$P(\mathscr{C})\geq 1-\frac{r/2}{\epsilon-r}=\frac{2\epsilon-3r}{2\epsilon-2r}.$$
\end{example}
Observe that $\Delta=[0,1]$ is equivalent to $\Delta=B(0,1/2)$, therefore $P(\mathscr{C})\geq \frac{4\epsilon-3}{4\epsilon-2}.$ Observe that the larger the confidence threshold is, the more likely a consensus can be achieved.
\section{Conclusion}
For the mixed model, all agents can decide their social relationship with each other, and the two selected social neighbors can decide their approach rate at all times. Some social relationships, such as a friendship, may not last for ever. The approach rate depicts the degree of cooperation among the two selected social neighbors. The larger the approach rate is, the more likely they cooperate. The variation of the social relation and the approach rate over time for the mixed model is much closer to real world circumstances. It turns out that asymptotic stability holds for the mixed model if the social graph is connected for infinitely many times and the two selected social neighbors approach each other at a rate no less than some positive number at all times. Also, we derive a lower bound for the probability of consensus, independent of the size and topology of a graph and the convergence parameter.

\end{document}